\documentclass[11pt,a4paper]{article}
\usepackage[T1]{fontenc}
\usepackage{lmodern}
\usepackage[a4paper,margin=2.5cm]{geometry}
\usepackage{amsmath,amssymb,amsthm}
\usepackage{array,booktabs}
\usepackage{fancyhdr}
\usepackage{xcolor}
\usepackage{pict2e}   % light-weight drawing, used instead of TikZ (fast on Overleaf)
\usepackage[colorlinks=true,linkcolor=blue!45!black,citecolor=blue!45!black,
            urlcolor=blue!45!black]{hyperref}

\newtheorem{theorem}{Theorem}[section]
\newtheorem{lemma}[theorem]{Lemma}
\newtheorem{proposition}[theorem]{Proposition}
\newtheorem{fact}[theorem]{Fact}
\newtheorem*{mainthm}{Main Theorem}
\theoremstyle{definition}
\newtheorem{definition}[theorem]{Definition}
\newtheorem{remark}[theorem]{Remark}
\numberwithin{equation}{section}

\newcommand{\R}{\mathbb{R}}
\newcommand{\N}{\mathbb{N}}
\newcommand{\Q}{\mathbb{Q}}
\newcommand{\Sph}{\mathbb{S}}
\newcommand{\B}{\mathcal{B}}
\newcommand{\M}{\mathbf{M}}
\newcommand{\D}{\mathcal{D}}
\newcommand{\pos}[1]{[#1]_+}
\newcommand{\sat}{\operatorname{sat}}
\newcommand{\dist}{\operatorname{dist}}
\newcommand{\nrm}[1]{\|#1\|_\infty}
\newcommand{\INC}{\mathsf{INC}}
\newcommand{\DEC}{\mathsf{DEC}}
\newcommand{\chainbox}[2]{\fbox{\parbox{0.5\textwidth}{\centering\strut#1\\
  {\itshape\color{blue!45!black}#2\strut}}}}

\title{Lyapunov stability of polynomial vector fields is undecidable}
\author{Milan Korda$^{1,2}$}

\date{18 September 2026}

\begin{document}
\maketitle
\footnotetext[1]{Faculty of Electrical Engineering, Czech Technical University in Prague, Technick\'a 2, CZ-16627 Prague, Czechia.}
\footnotetext[2]{CNRS; LAAS; Universit\'e de Toulouse, 7 avenue du colonel Roche, F-31400 Toulouse, France.}
\setcounter{footnote}{2}

\begin{abstract}
We show that there are integers $N$ and odd $D$ such that no algorithm can decide,
from the rational coefficients of a homogeneous polynomial vector field
$F\colon\R^N\to\R^N$ of degree $D$, whether the origin is Lyapunov stable for
$\dot Y=F(Y)$. This proves a conjecture of V.\,I. Arnold. We provide a Lean formalization of the proof.

The smallest dimension $N$ for which we were able to prove undecidability is $N = 5$. An analogous undecidability result holds for global asymptotic stability and global exponential stability for non-homogenous vector fields. For dimension $N=2$ and homogenous vector fields, we establish decidability for all commonly used stability notions, building heavily on existing results.
\end{abstract}

\tableofcontents

\section{Introduction}\label{sec:intro}

The central notion of this work is the following:

\begin{definition}[Lyapunov stability]\label{def:stab}
Let $F\colon\R^N\to\R^N$ be locally Lipschitz with $F(0)=0$. The origin is
\emph{Lyapunov stable} for $\dot Y=F(Y)$ if for every $\epsilon>0$ there is
$\eta>0$ such that every solution with $|Y(0)|<\eta$ exists for all $t\ge0$ and
satisfies $|Y(t)|<\epsilon$ for all $t\ge0$.
\end{definition}
\noindent Lyapunov's first method settles the stability of an equilibrium whenever the
linearization has no eigenvalue on the imaginary axis. In the remaining cases
stability is decided by higher-order terms, and Arnold showed that no
\emph{algebraic} criterion can settle it in general~\cite{Arn70}. In his 1976 list
of problems \cite{Arn76} he asked for the algorithmic version (quoted as in~\cite{arnold_symposium}):

%(quoted as
%in~\cite{AP13}):

\begin{quote} \it Let a vector field be given by polynomials of a fixed degree, with rational coefficients. Does an algorithm exist, allowing to decide, whether the stationary point is stable?
\end{quote}

Later he communicated a more detailed version of his question~\cite{arnold_symposium}.

\begin{quote} \it
In my problem the coefficients of the polynomials of known degree and of a known number of variables are written on the tape of the standard Turing machine in the standard order and in the standard representation.

The problem is whether there exists an algorithm (an additional text for the machine independent of the values of the coefficients) such that it solves the stability problem for the stationary point at the origin (i.e., always stops giving the answer “stable” or “unstable”).

I hope, this algorithm exists if the degree is one. It also exists when the dimension is one. My conjecture has always been that there is no algorithm for some sufficiently high degree and dimension, perhaps for dimension 3 and degree 3 or even 2. I am less certain about what happens in dimension 2.
    
\end{quote}

%\begin{quote}\itshape
%Is the stability problem for stationary points algorithmically decidable?
%[\dots] Let a vector field be given by polynomials of a fixed degree, with
%rational coefficients. Does an algorithm exist, allowing to decide, whether the
%stationary point is stable?
%\end{quote}

Currently, undecidability is known for systems built from non-polynomial elementary
functions~\cite{dCD91}. For polynomial vector fields, deciding local or global asymptotic stability is
already strongly NP-hard for cubic vector fields~\cite{AP13}. The purpose of this
paper is to give a self-contained proof of the following
statement that confirms Arnold's conjecture.

\begin{mainthm}
There are integers $N$ and odd $D$ such that no algorithm decides, given the
rational coefficients of a homogeneous polynomial vector field
$F\colon\R^N\to\R^N$ of degree $D$, whether the origin is Lyapunov stable for
$\dot Y=F(Y)$.
\end{mainthm}

We use homogenous in the standard sense, i.e., $F$ is homogenous of degree $D$ if $F(\lambda Y) = \lambda^D F(Y)$ for any $\lambda \in \R$. 

Since $N$ and $D$ are fixed, a vector field $F$ is described by a finite list of rational numbers,
and ``algorithm'' has its usual meaning. 

The following three remarks make the Main theorem more accurate and discuss its generalizations. Proofs of the claims made therein can be found in the supplementary material~\cite{arnold_supplement}.

\begin{remark}[Minimal dimension $N$]
 The minimal dimension $N$ for which we were able to prove undecidability is $N = 5$.  Proving or disproving it in dimension three or four remains an open problem. Arnold believed that undecidability may occur already in dimension three or perhaps even two.     
\end{remark}
\begin{remark}[Dimensions $N = 2$ and $N=1$]
We were able to establish decidability in dimension two for homogenous  vector fields and for all commonly used stability notions, building heavily on existing results. The general non-homogenous case remains open in $N=2$. The case $N = 1$ is decidable.
\end{remark}

\begin{remark}[Global asymptotic and exponential stability]
We also obtained a result analogous to the Main theorem for Global asymptotic stability (GAS) and Global exponential stabilitiy (GES) of equilibria, proving that GAS and GES are undecidable for polynomial vector fields. The construction relies on the vector field being non-homogenous; the homogenous case remains open for GAS and GES.
\end{remark}

\subsection{Lean formalization}
We carried out a Lean formalization of the Main theorem. The result can be verified in 
\begin{center}
\url{https://homepages.laas.fr/mkorda/arnold_lean.zip}. 
\end{center}
The subsequent claims in Remarks 1-4 have not been Lean-formalized, at the  time of writing this article.

\subsection*{Proof outline}

The proof is a computable reduction: for
a fixed register machine $U$ with an undecidable halting problem
(Fact~\ref{fact:minsky}) we construct, algorithmically in $n\in\N$, a homogeneous
field $F_n$ in this class such that
\begin{equation}\label{eq:reduction}
  U\text{ halts on input }n\iff 0\text{ is Lyapunov stable for }\dot Y=F_n(Y).
\end{equation}
The construction of $F_n$ never runs $U$; the halting time only appears in the
\emph{proof} of~\eqref{eq:reduction}.

\paragraph{Homogeneity turns stability into a uniform growth bound.} If $F$ is
homogeneous of degree $D$ and $Y(t)$ is a solution, then so is
$Y_\lambda(t):=\lambda\,Y(\lambda^{D-1}t)$ for every $\lambda>0$. This scaling
symmetry makes the local notion of stability global:

\begin{lemma}\label{lem:scaling}
For a homogeneous polynomial field $F$, the origin is Lyapunov stable if and only
if there is $C<\infty$ such that every solution exists for all $t\ge0$ and
satisfies $|Y(t)|\le C|Y(0)|$.
\end{lemma}

\begin{proof}
``If'' is clear with $\eta=\epsilon/C$. Conversely, let $\eta$ correspond to
$\epsilon=1$. Given $Y(0)\ne0$, put $\lambda=\eta/(2|Y(0)|)$. The solution
$Y_\lambda$ starts at norm $\eta/2$, hence exists forever and stays in the unit
ball; thus $Y$ exists forever and $|Y(s)|<1/\lambda=(2/\eta)|Y(0)|$.
\end{proof}

So for homogeneous fields, stability is a statement about \emph{all} trajectories,
uniformly: no direction may be amplified by more than a fixed factor. In our
construction every initial state of an auxiliary (non-homogeneous) system will
correspond to a direction in $\R^N$. The ``halting $\Rightarrow$ stable'' half
of~\eqref{eq:reduction} will therefore require estimates that hold uniformly over
\emph{every} real initial state of that system, including states that encode
nothing meaningful. Most design choices below exist to provide this uniformity.

\paragraph{The two halves of~\eqref{eq:reduction} are asymmetric.} Stability
needs a bound on all trajectories; instability needs a single bad family. Hence in
the non-halting case we only ever have to follow \emph{one} carefully initialized
trajectory, while in the halting case we must control arbitrary ones. Along the
chain of systems in Figure~\ref{fig:chain} we therefore keep track of two
properties:
\begin{itemize}
\item[(H)] if $U$ halts on $n$, the output is bounded uniformly over all initial
  states;
\item[(N)] if $U$ does not halt on $n$, one explicit trajectory produces infinite
  output.
\end{itemize}

\begin{figure}[t]
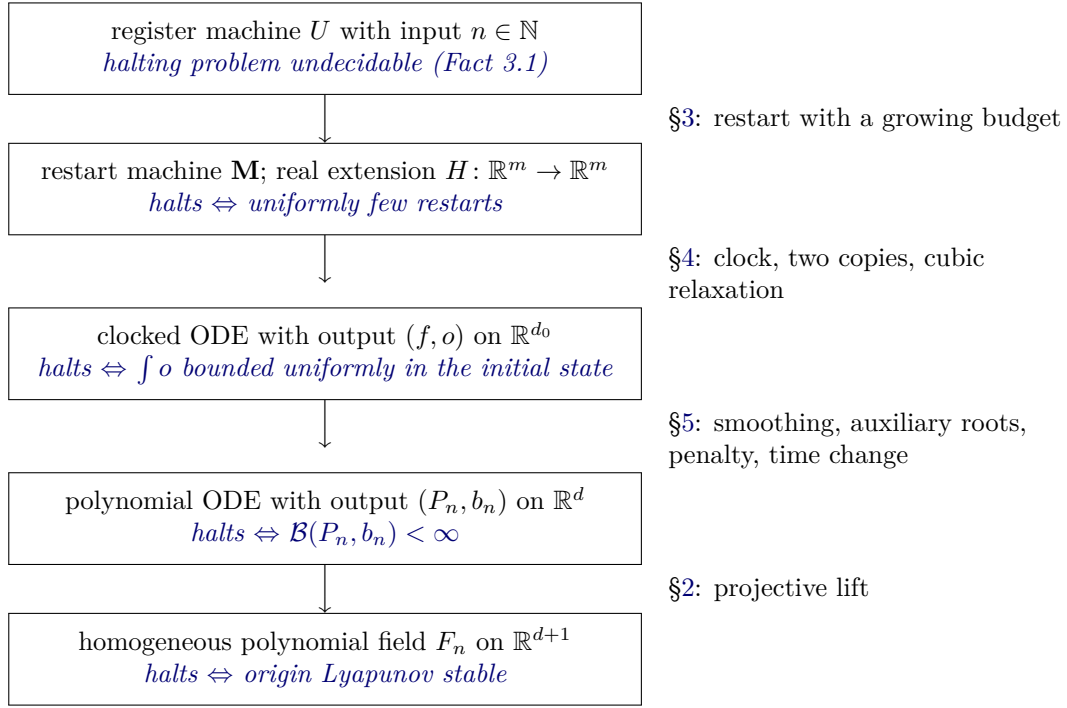

\centering\small
\setlength{\fboxsep}{5pt}
\begin{tabular}{@{}c@{\hspace{1em}}>{\raggedright\arraybackslash}p{0.33\textwidth}@{}}
\chainbox{register machine $U$ with input $n\in\N$}
         {halting problem undecidable (Fact~\ref{fact:minsky})} & \\
$\Big\downarrow$ & \S\ref{sec:machine}: restart with a growing budget\\
\chainbox{restart machine $\M$; real extension $H\colon\R^m\to\R^m$}
         {halts $\Leftrightarrow$ uniformly few restarts} & \\
$\Big\downarrow$ & \S\ref{sec:monitor}: clock, two copies, cubic relaxation\\
\chainbox{clocked ODE with output $(f,o)$ on $\R^{d_0}$}
         {halts $\Leftrightarrow$ $\int o$ bounded uniformly in the initial state} & \\
$\Big\downarrow$ & \S\ref{sec:roots}: smoothing, auxiliary roots, penalty,
                   time change\\
\chainbox{polynomial ODE with output $(P_n,b_n)$ on $\R^{d}$}
         {halts $\Leftrightarrow$ $\B(P_n,b_n)<\infty$} & \\
$\Big\downarrow$ & \S\ref{sec:lift}: projective lift\\
\chainbox{homogeneous polynomial field $F_n$ on $\R^{d+1}$}
         {halts $\Leftrightarrow$ origin Lyapunov stable} &
\end{tabular}
\caption{The reduction $n\mapsto F_n$. Boxes are objects computed from $n$, arrows
are constructions, and the italic line in each box is the property proved about
it.}
\label{fig:chain}
\end{figure}

\subsection*{Roadmap}

%In one sentence: an ODE repeatedly runs a fixed computation with growing time
%budgets and emits output at each restart; a penalty keeps this working when
%auxiliary variables are initialized wrongly; and a homogeneous lift turns total
%output into logarithmic radial growth.

The reduction is a composition of
four explicit constructions (Figure~\ref{fig:chain}). We present them top-down:
first the last step (the lift), which explains \emph{what} has to be built, and
then the three steps that build it.

The central notion is the \emph{accumulated output} $\B(P,b)$ of a vector field
$P$ with a scalar ``output'' $b$: the supremum, over all initial states and all
finite times, of $\int b$ along the trajectory (Definition~\ref{def:B}).
Section~\ref{sec:lift} shows that a homogeneous lift of $(P,b)$ is stable exactly
when $\B(P,b)<\infty$. Sections~\ref{sec:machine} to~\ref{sec:roots} then build a
polynomial pair $(P_n,b_n)$ with $\B(P_n,b_n)<\infty$ exactly when $U$ halts on
$n$. Table~\ref{tab:obstacles} lists the obstacles met along the way and the device
that removes each of them.

\paragraph{Reading guide.} Sections~\ref{sec:lift} and~\ref{sec:monitor} pertain to ODEs. Section~\ref{sec:machine} contains all the computability theory that is
needed, which is one standard fact. Section~\ref{sec:roots} is elementary real
algebra. Section~\ref{sec:proof} assembles the proof and summarizes the
construction.

\subsection*{Conventions}

On machine states we use the sup norm $\nrm{\cdot}$; elsewhere $|\cdot|$ is the
Euclidean norm. We write
\[
  \pos x=\max\{x,0\}=\tfrac12\bigl(x+|x|\bigr),\qquad
  \sat(x)=\min\{\max\{x,0\},1\}=\tfrac12\bigl(|x|-|x-1|+1\bigr).
\]
From ODE theory we use only local existence and uniqueness for fields that are
continuous in time and locally Lipschitz in the state, and the continuation
criterion: a maximal solution defined on $[0,t^*)$ with $t^*<\infty$ is
unbounded~\cite[Theorems 2.2, 2.13, Corollary 2.16]{Tes12}. A \emph{solution
segment} is the restriction of a solution to a compact interval contained in its
interval of existence.

\begin{table}[t]
\centering\small
\begin{tabular}{@{}>{\raggedright\arraybackslash}p{0.34\textwidth}>{\raggedright\arraybackslash}p{0.45\textwidth}l@{}}
\toprule
Obstacle & Device & Where\\
\midrule
an initial machine state may be garbage
  & restart with a growing budget; the output counts restarts
  & Lemma~\ref{lem:H}\\
the copies $a,b$ may start arbitrarily far from their targets
  & cubic relaxation: the error after relaxation does not depend on the initial
    distance
  & Lemma~\ref{lem:cubic}\\
infinitely many steps must be simulated with bounded total error
  & relaxation gains grow with a clock $\theta$, so errors are summable, uniformly
    in $\theta(0)$
  & Lemma~\ref{lem:summable}\\
the oscillator may start with any amplitude
  & output switched off unless the squared amplitude is in $(\frac12,\frac32)$
  & \eqref{eq:output}\\
absolute values are not polynomial
  & smoothing with a precision that grows with $\theta$ and $|x|$
  & Lemma~\ref{lem:smooth}\\
square roots are not polynomial
  & auxiliary $w_j$ with $w_j^4=\D_j$, protected by a first integral
  & Lemma~\ref{lem:integrals}\\
auxiliary variables may start with wrong values
  & quadratic penalty in the output
  & Proposition~\ref{prop:G}\\
denominators
  & positive change of time
  & Lemma~\ref{lem:timechange}\\
\bottomrule
\end{tabular}
\caption{Obstacles to uniformity and the devices that remove them.}
\label{tab:obstacles}
\end{table}

\section{Stability as a bound on accumulated output: the projective lift}
\label{sec:lift}

\begin{definition}[Accumulated output]\label{def:B}
Let $P\colon\R^d\to\R^d$ be locally Lipschitz and $b\colon\R^d\to\R$ continuous.
With $x(t;x_0)$ the maximal solution of $\dot x=P(x)$, defined for
$0\le t<t^+(x_0)$, put
\begin{equation}\label{eq:B}
  \B(P,b)=\sup_{x_0\in\R^d}\ \sup_{0\le t<t^+(x_0)}\int_0^t b\bigl(x(s;x_0)\bigr)\,ds
  \ \in[0,\infty].
\end{equation}
\end{definition}

Three features of this definition matter later. (i) The supremum over $x_0$ is
essential: for $\dot x=-x$ and $b=x^2$ every trajectory has finite total output
$x_0^2/2$, yet $\B=\infty$. (ii) Only upper bounds on $b$ are relevant: $b$ may be
very negative somewhere, and we exploit this by subtracting a penalty in
Section~\ref{sec:roots}. (iii) Only finite solution segments enter, so $P$ need not
be complete; and $\B\ge0$ because $t=0$ is allowed.

\subsection*{Polar decomposition}

Let $F\colon\R^N\to\R^N$ be homogeneous of degree $D\ge1$. On the unit sphere
$\Sph=S^{N-1}$ define the \emph{radial rate} and the \emph{tangential field}
\[
  \lambda_F(\omega)=\langle F(\omega),\omega\rangle,\qquad
  F^\top(\omega)=F(\omega)-\lambda_F(\omega)\,\omega .
\]
For a solution $Y(t)\neq0$ write $Y=r\omega$ with $r=|Y|$. Since
$\dot Y=r^DF(\omega)$, we get $\dot r=r^D\lambda_F(\omega)$ and
$\dot\omega=r^{D-1}F^\top(\omega)$. In the new time $s$ with $ds=r^{D-1}dt$,
\begin{equation}\label{eq:polar}
  \frac{d\omega}{ds}=F^\top(\omega),\qquad \frac{d\log r}{ds}=\lambda_F(\omega).
\end{equation}
Thus the direction follows an autonomous flow $\Phi^s$ on the compact sphere,
which is complete, and $\log r$ integrates the radial rate along it.

\begin{lemma}\label{lem:polar}
The origin is stable for $\dot Y=F(Y)$ if and only if
\[
  M_F:=\sup_{\omega\in\Sph,\ s\ge0}\int_0^s\lambda_F(\Phi^{s'}\omega)\,ds'<\infty .
\]
\end{lemma}

\begin{proof}
Let $Y(0)=r_0\omega_0\neq0$. As $F(0)=0$, uniqueness gives $Y(t)\ne0$ on the
maximal interval, and by~\eqref{eq:polar}
\begin{equation}\label{eq:radius}
  r(t)=r_0\exp\int_0^{s(t)}\lambda_F(\Phi^{s'}\omega_0)\,ds',\qquad
  s(t)=\int_0^tr^{D-1}.
\end{equation}
If $M_F<\infty$, then $r(t)\le e^{M_F}r_0$. Solutions are therefore bounded,
hence global, and $|Y(0)|<e^{-M_F}\epsilon$ implies $|Y(t)|<\epsilon$ for all
$t\ge0$.

If $M_F=\infty$, choose $\omega_j$ and $s_j$ with
$L_j:=\int_0^{s_j}\lambda_F(\Phi^{s'}\omega_j)\,ds'\to\infty$, and start at
$Y_j(0)=e^{-L_j}\omega_j$. As long as $s(t)\le s_j$, \eqref{eq:radius} keeps
$r(t)$ between two positive constants, so the solution cannot blow up and
$ds/dt=r^{D-1}$ stays bounded below. Hence $s(t)$ reaches $s_j$ at a finite time,
where $r=e^{-L_j}e^{L_j}=1$. Initial points converging to $0$ thus reach the unit
sphere, and the origin is unstable.
\end{proof}

In the language of Lemma~\ref{lem:scaling}, one can take $C=e^{M_F}$.
Stability of a homogeneous field is thus a statement about an output, the radial
rate, integrated along a flow on a compact manifold and bounded uniformly in the
initial point. We now go the other way and realize a prescribed pair $(P,b)$ in
this form.

\subsection*{Prescribing the flow on the sphere and the radial rate}

Fix $P\in\Q[x]^d$, $b\in\Q[x]$, and an odd integer $\ell$ with $\ell>\deg P$ and
$\ell-1>\deg b$ (the zero polynomial has degree $0$). For
$Y=(y,z)\in\R^d\times\R$ set
\[
  A(Y)=\bigl(z^\ell P(y/z),\,0\bigr)\in\R^{d+1},\qquad \beta(Y)=z^{\ell-1}b(y/z),
\]
understood as polynomials after expansion. They are homogeneous of degrees $\ell$
and $\ell-1$, and they vanish on $\{z=0\}$ because the degree inequalities are
strict. Define
\begin{equation}\label{eq:F}
  F(Y)=|Y|^2A(Y)-\langle Y,A(Y)\rangle\,Y+|Y|^2\beta(Y)\,Y .
\end{equation}
This is a homogeneous polynomial field of odd degree $\ell+2$ with rational
coefficients. The first two terms are $|Y|^2$ times the component of $A(Y)$
orthogonal to $Y$ (the factor $|Y|^2$ only makes this projection polynomial), and
the last term is radial. Hence on the unit sphere
\begin{equation}\label{eq:Fsphere}
  F^\top(\omega)=A(\omega)-\langle\omega,A(\omega)\rangle\,\omega,\qquad
  \lambda_F(\omega)=\beta(\omega).
\end{equation}

The \emph{equator} $\Sph\cap\{z=0\}$ consists of equilibria of $F^\top$ at which
$\lambda_F=0$; in fact every point of the hyperplane $\{z=0\}$ is an equilibrium
of $F$. On the open upper hemisphere use the central projection $x=y/z$, with
inverse $\omega=(x,1)/\sqrt{1+|x|^2}$. Writing $\omega'=d\omega/ds$ and
$c=\langle\omega,A(\omega)\rangle$, equation~\eqref{eq:Fsphere} gives
$(y',z')=(z^\ell P(x)-cy,\,-cz)$, hence
\[
  \frac{dx}{ds}=\frac{y'z-yz'}{z^2}=z^{\ell-1}P(x),\qquad
  \lambda_F(\omega)=z^{\ell-1}b(x).
\]
Since $\ell-1$ is even, $z^{\ell-1}=(1+|x|^2)^{-(\ell-1)/2}>0$, and in the time
$\tau$ with $d\tau=z^{\ell-1}ds$,
\begin{equation}\label{eq:chart}
  \frac{dx}{d\tau}=P(x),\qquad \int\lambda_F(\omega)\,ds=\int b(x)\,d\tau .
\end{equation}
The lower hemisphere behaves in the same way, since $F(-Y)=-F(Y)$.

\begin{remark}[Poincar\'e compactification]
On the open upper hemisphere, $F^\top$ is a positive multiple of the Poincar\'e
compactification of $P$~\cite[Ch.~5]{DLA06}. The usual compactification multiplies
by $z^{\deg P-1}$ to obtain a nontrivial flow at infinity. Our larger power turns
all points at infinity into equilibria, so a trajectory of $P$ escaping in finite
time becomes a trajectory on the sphere that creeps towards the equator forever.
The radial direction, which the compactification discards, stores the
accumulated output.
\end{remark}

\begin{proposition}[Projective lift]\label{prop:lift}
For $F$ as in~\eqref{eq:F}, the origin is Lyapunov stable if and only if
$\B(P,b)<\infty$.
\end{proposition}

\begin{proof}
By Lemma~\ref{lem:polar} it suffices to show $M_F=\B(P,b)$. Orbits on the equator
contribute $0$, and the two open hemispheres are invariant because the equator
consists of equilibria. By~\eqref{eq:chart}, a trajectory segment of $\Phi$ on
$[0,s]$ in an open hemisphere corresponds to a solution segment of $P$ on
$[0,\tau(s)]$, with $\tau(s)=\int_0^sz^{\ell-1}<\infty$ and the same output
integral. Conversely, a solution segment $x\colon[0,T]\to\R^d$ of $P$ lifts to
$\omega(\tau)=(x(\tau),1)/\sqrt{1+|x(\tau)|^2}$, traversed in the finite
$s$-time $\int_0^T(1+|x|^2)^{(\ell-1)/2}d\tau$. Both suprema are therefore taken
over the same set of numbers, together with $0$.
\end{proof}

From now on the task is to build polynomial pairs $(P_n,b_n)$ with
$\B(P_n,b_n)<\infty$ exactly when $U$ halts on $n$.

\section{The discrete skeleton: restarting a register machine}\label{sec:machine}

\subsection*{Register machines}

A register machine $U$ has registers $r\in\N^k$, instruction labels $1,\dots,S$
and a halting label $S+1$. The instruction at a label is either $\INC(\nu;q')$,
which increments $r_\nu$ and jumps to $q'$, or $\DEC(\nu;q_0,q_+)$, which jumps to
$q_0$ if $r_\nu=0$ and otherwise decrements $r_\nu$ and jumps to $q_+$. On input
$n\in\N$ the machine starts at label $1$ with registers $r_n=(n,0,\dots,0)$; it
\emph{halts on $n$} if it reaches the label $S+1$, and the number $T$ of steps
needed is the \emph{halting time}. A register machine is thus nothing but a
discrete dynamical system on $\N^k\times\{1,\dots,S+1\}$ whose rules use only
$\pm1$ and a test for zero.

\begin{fact}[Turing, Minsky]\label{fact:minsky}
There is a register machine $U$ such that no algorithm decides, given $n\in\N$,
whether $U$ halts on input $n$.
\end{fact}

\begin{remark}
Here is why. Take a universal Turing machine, whose halting set is
undecidable~\cite{Tur36}, and store the tape on either side of the head as two
stacks, each kept in a register as a numeral in a base $B$ larger than the
alphabet~\cite{Min61,Min67}. Pushing a symbol is $N\mapsto BN+\text{digit}$ and
popping is division with remainder by $B$; both are loops of increments,
decrements and zero tests with an auxiliary register, while the current symbol
and the machine state are kept in the label. If the first register initially holds
the code $n$ of an input word, read as a base-$B$ numeral, then $U$ halts on $n$ if
and only if the Turing machine halts on the word. Codes are easy to recognize, so
halting of $U$ is undecidable.
\end{remark}

From now on $U$, $k$ and $S$ are fixed, and $n$ is the only input.

\subsection*{Restarting with growing budgets}

Enlarge the state to $a=(h,c,r,q)$, where $h$ is a \emph{budget}, $c$ a
\emph{countdown} of remaining steps, and the label $q\in\{0,1,\dots,S+1\}$ has one
new value $0$, called \emph{restart}. One step of the restarting machine $\M$ is:
\begin{itemize}
\item $q=0$ (restart): $a\mapsto(h^++1,\,h^++1,\,r_n,\,1)$, where
  $h^+=\max\{h,0\}$;
\item $1\le q\le S$ and $c\le0$ (budget exhausted): $a\mapsto(h,\,c-1,\,r,\,0)$;
\item $1\le q\le S$ and $c\ge1$: $a\mapsto(h,\,c-1,\,r',\,q')$, where $(r',q')$ is
  one step of $U$;
\item $q=S+1$ (halted): $a\mapsto a$.
\end{itemize}
A restart erases everything except the budget, loads the correct input, and grants
a budget one larger than before. From the restart state $(0,0,r_n,0)$, the machine
$\M$ therefore runs $U$ with budgets $1,2,3,\dots$. If $U$ halts after $T$ steps,
the attempt with budget $T$ succeeds and no further restart occurs; if $U$ never
halts, $\M$ restarts infinitely often. Thus
\begin{equation}\label{eq:restarts}
  U\text{ halts on }n\iff\M\text{ restarts only finitely often.}
\end{equation}
The same is true from every integer state, and uniformly so: after the $j$-th
restart the budget is at least $j$, so after at most $T$ restarts the budget
suffices and the run halts. This is why our ODE will count restarts rather than
steps. The number of steps spent at ordinary labels also detects non-halting, but
it is not bounded uniformly over initial states: a garbage countdown can be
arbitrarily large. The rest of the construction preserves the uniform bound on
restarts when $\M$ is embedded in a flow with arbitrary real initial data.

\subsection*{Extension to real states}

We extend $\M$ to a map $H\colon\R^m\to\R^m$, $m=k+3$, using only rational
constants, $+$, $\times$ and $|\cdot|$. We use the bump and the zero test
\[
  \chi(t)=\sat(2-8|t|),\qquad Z(t)=\sat\bigl(\tfrac32-2t\bigr):
\]
$\chi$ equals $1$ for $|t|\le\frac18$ and vanishes for $|t|\ge\frac14$, while $Z$
equals $1$ for $t\le\frac14$ and $0$ for $t\ge\frac34$. For every label define a
branch map $M_i\colon\R^m\to\R^m$:
\[
  M_0(a)=\bigl(\pos h+1,\,\pos h+1,\,r_n,\,1\bigr),\qquad
  M_{S+1}(a)=(h,\,c,\,r,\,S+1),
\]
and, abbreviating $\zeta=Z(c)$ and $\zeta_\nu=Z(r_\nu)$, for the instruction at
label $i$:
\begin{align*}
  \INC(\nu;q')&:\ M_i(a)=\bigl(h,\,c-1,\,r+(1-\zeta)e_\nu,\,(1-\zeta)q'\bigr),\\
  \DEC(\nu;q_0,q_+)&:\ M_i(a)=\bigl(h,\,c-1,\,r-(1-\zeta)(1-\zeta_\nu)e_\nu,\,
     (1-\zeta)(\zeta_\nu q_0+(1-\zeta_\nu)q_+)\bigr).
\end{align*}
Finally
\begin{equation}\label{eq:H}
  H(a)=M_{S+1}(a)+\sum_{i=0}^{S}\chi(q-i)\bigl(M_i(a)-M_{S+1}(a)\bigr).
\end{equation}
On integer states $H=\M$. A state whose label is far from $\{0,\dots,S\}$ is
simply declared halted. Only $r_n$ depends on the input.

The fine structure of $H$ plays no role. The proof uses only the following three
properties, and the first two hold on all of $\R^m$.

\begin{lemma}\label{lem:H}
The map $H$ has the following properties.
\begin{itemize}
\item[(a)] \emph{The budget never decreases:}
  $H_h(a)=h+\chi(q)\bigl(1+\pos{-h}\bigr)\ge h$ for every $a\in\R^m$.
\item[(b)] \emph{Restarts are clean:} if $|q|\le\frac18$, then $H(a)=M_0(a)$, and
  $M_0$ is $1$-Lipschitz on $\R^m$ for $\nrm{\cdot}$.
\item[(c)] \emph{Exactness near honest states:} call $a^*$ \emph{regular} if
  $q^*\in\{0,\dots,S+1\}$, $r^*\in\N^k$, and $c^*\in(-\infty,0]\cup[1,\infty)$
  whenever $1\le q^*\le S$; the rules of $\M$ apply verbatim to such states, with
  $h^*,c^*$ real. For each regular $a^*$ there is a $1$-Lipschitz map $F_{a^*}$
  with $F_{a^*}(a^*)=\M(a^*)$ and $H=F_{a^*}$ on the ball
  $\nrm{a-a^*}\le\frac18$.
\end{itemize}
\end{lemma}

\begin{proof}
(a) Only $M_0$ changes $h$, and $\pos h+1=h+1+\pos{-h}$. (b) For $|q|\le\frac18$
the bump $\chi(q)$ equals $1$ and all $\chi(q-i)$, $i\ge1$, vanish. (c) On the
ball, $\chi(q-q^*)=1$ and all other bumps vanish. Moreover $c$ and every $r_\nu$
avoid $(\frac18,\frac78)$, so the zero tests are constant there. Hence $H$
coincides on the ball with $M_0$ if $q^*=0$, with $M_{S+1}$ if $q^*=S+1$, and
otherwise with a map that translates $(h,c,r)$ by a fixed vector and sets $q$ to a
constant.
\end{proof}

Properties (a) and (b) make the budget a Lyapunov-like quantity that is valid
everywhere and make every restart exact, whatever garbage preceded it. Together
with the local exactness (c), this is all the robustness we need.

\section{Robust simulation by a clocked ODE}\label{sec:monitor}

\subsection*{The equations}

The \emph{monitor} has state $x=(\theta,u,v,a,b)\in\R\times\R^2\times\R^m\times\R^m$,
so $x\in\R^{d_0}$ with $d_0=3+2m$, and reads
\begin{equation}\label{eq:monitor}
\begin{aligned}
  \dot\theta&=1,\qquad \dot u=-v,\qquad \dot v=u,\\
  \dot b&=\kappa(\theta)\,\pos u\,\bigl(H(a)-b\bigr)^3+\varphi_b(t),\\
  \dot a&=\kappa(\theta)\,\pos{-u}\,(b-a)^3+\varphi_a(t),
\end{aligned}
\qquad \kappa(\theta)=K(1+\theta^2)^2,\quad K=2^{20},
\end{equation}
with cubes taken componentwise. The \emph{forcing} $\varphi=(\varphi_a,\varphi_b)$
is a continuous function of time. The monitor proper has $\varphi=0$, and we denote
its vector field by $f$. Forcing is a bookkeeping device: in
Section~\ref{sec:roots} every approximation error will be treated as forcing. The
output is
\begin{equation}\label{eq:output}
  o(x)=\mu(u^2+v^2)\,\pos u\,\chi(4a_q),\qquad \mu(s)=\bigl[1-2|s-1|\bigr]_+ .
\end{equation}

\begin{figure}[htb]
\centering\small
\setlength{\unitlength}{1cm}
\begin{picture}(11,4.6)(-5.5,-2.3)
  % colours are set inside each \put, so that no stray space shifts the picture
  \put(-2.4,0){\color{gray}\vector(1,0){4.9}}
  \put(0,-2.2){\color{gray}\vector(0,1){4.3}}
  \put(2.6,0){\makebox(0,0)[l]{$u$}}
  \put(0,2.25){\makebox(0,0)[b]{$v$}}
  \linethickness{1.2pt}
  \put(0,0){\color{blue!45!black}\arc[-90,90]{1.5}}
  \put(0,0){\color{gray!70}\arc[90,270]{1.5}}
  \linethickness{2.6pt}
  \put(0,0){\color{blue!45!black!45}\arc[-60,60]{1.72}}
  \put(0,0){\color{gray!45}\arc[120,240]{1.72}}
  \thinlines
  \put(0.667,1.832){\vector(-940,342){0.38}}
  \put(-0.667,-1.832){\vector(940,-342){0.38}}
  \put(0,-1.5){\circle*{0.13}}
  \put(0.2,-2.0){\makebox(0,0)[l]{turn starts, $\psi=-\frac\pi2$}}
  \put(2.0,0.95){\makebox(0,0)[l]{\shortstack[l]{compute: $a$ frozen,\\
     $b\to H(a)$, output paid}}}
  \put(-2.0,0.95){\makebox(0,0)[r]{\shortstack[r]{copy: $b$ frozen,\\
     $a\to b$}}}
\end{picture}
\caption{The clock $(u,v)=\rho(\cos\psi,\sin\psi)$ turns counterclockwise with
$\dot\psi=1$. The outer arcs mark the middle thirds of the two half-turns, used in
the error estimates.}
\label{fig:clock}
\end{figure}
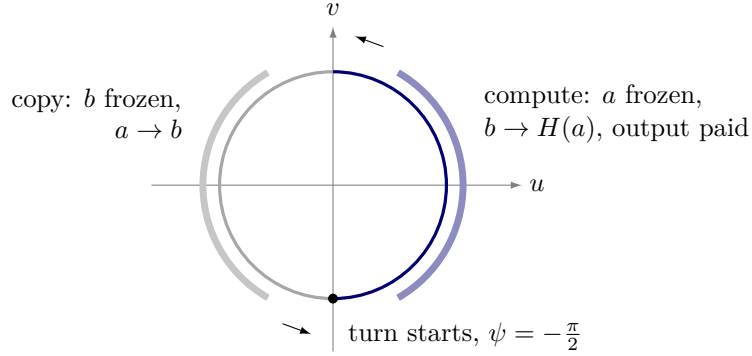

Write $(u,v)=\rho(\cos\psi,\sin\psi)$. The amplitude $\rho$ is constant and
$\dot\psi=1$ (Figure~\ref{fig:clock}). Each ingredient has a specific purpose.
\begin{itemize}
\item \emph{Sample and hold.} During the half-turn $u>0$ (\emph{compute}) the copy
  $a$ is frozen and each $b_i$ is pulled towards $H_i(a)$. During the half-turn
  $u<0$ (\emph{copy}) $b$ is frozen and $a$ is pulled towards $b$. One full turn
  therefore performs $a\mapsto H(a)$ approximately. This device, which lets a flow
  iterate a map, goes back to Branicky~\cite{Bra95} and is a standard way to
  simulate maps by ODEs~\cite{GCB08,BGP17}. What is new here is that every
  estimate must be uniform over all initial states.
\item \emph{Cubic relaxation.} A linear relaxation $\dot y=\gamma(z-y)$ leaves the
  error $e^{-\Gamma}|y(0)-z|$, which is not small uniformly in $y(0)$. The cubic
  relaxation forgets its initial condition completely: its error is at most
  $(2\Gamma)^{-1/2}$ (Lemma~\ref{lem:cubic}).
\item \emph{Growing gain.} The error made in a half-turn is
  $O\bigl((1+\theta^2)^{-1}\bigr)$, which is summable. A trajectory can therefore
  follow infinitely many steps of $\M$ with a total error fixed in advance,
  whatever the initial value of $\theta$.
\item \emph{Output.} Output is paid only during compute half-turns and only while
  $a_q$ is within $\frac1{16}$ of the restart label. The factor $\mu$ restricts it
  to amplitudes with $\frac12<\rho^2<\frac32$. Below this window the clock is too
  weak for our error estimates; above it, a single turn could pay arbitrarily much.
\end{itemize}

\subsection*{Cubic relaxation}

\begin{lemma}\label{lem:cubic}
Let $\gamma,z,\varphi$ be continuous on $[t_0,t_1]$ with $\gamma\ge0$, let $y$
solve $\dot y=\gamma(t)\bigl(z(t)-y\bigr)^3+\varphi(t)$, and put
$\Gamma=\int_{t_0}^{t_1}\gamma$ and $\Phi=\int_{t_0}^{t_1}|\varphi|$.
\begin{itemize}
\item[(i)] If $|z(t)-z_*|\le\varrho$ on $[t_0,t_1]$ and $\Gamma>0$, then
  $|y(t_1)-z_*|\le\varrho+(2\Gamma)^{-1/2}+\Phi$.
\item[(ii)] If $z(t)\ge z_*$ on $[t_0,t_1]$ and $\Gamma>0$, then
  $y(t_1)\ge z_*-(2\Gamma)^{-1/2}-\Phi$.
\item[(iii)] $|y(t)|\le\max\{|y(t_0)|,\max|z|\}+\Phi$ on $[t_0,t_1]$; in
  particular $y$ exists on all of $[t_0,t_1]$.
\end{itemize}
The bounds (i) and (ii) do not depend on $y(t_0)$.
\end{lemma}

\begin{proof}
Let $\tilde y$ solve the unforced equation with $\tilde y(t_0)=y(t_0)$. The drift
$\gamma(z-y)^3$ is nonincreasing in $y$, so $|y-\tilde y|$ grows at most at rate
$|\varphi|$, and $|y-\tilde y|\le\Phi$. It remains to treat $\varphi=0$. For (i),
the interval $[z_*-\varrho,z_*+\varrho]$ is forward invariant, since the drift
points inwards at its ends. Above it, the excess $w=\tilde y-z_*-\varrho>0$
satisfies $\dot w\le-\gamma w^3$, that is,
$\frac{d}{dt}w^{-2}\ge2\gamma$. So either $w$ becomes nonpositive and stays so, or
$w(t_1)^{-2}\ge2\Gamma$. The same argument applies below the interval. Part (ii)
is the one-sided version, and (iii) follows from the invariance of
$[-R,R]$, $R=\max\{|y(t_0)|,\max|z|\}$.
\end{proof}

\subsection*{One turn of the clock}

If $\rho^2\notin(\frac12,\frac32)$ the output vanishes identically, so assume
$\frac12<\rho^2<\frac32$. A \emph{turn} is a time interval $[t_j,t_j+2\pi]$
starting when $\psi\equiv-\frac\pi2\pmod{2\pi}$. Its first half is the compute
half-turn and its second half the copy half-turn. On the middle third of each
half-turn ($|\psi|\le\frac\pi3$, respectively $|\psi-\pi|\le\frac\pi3$) the active
weight $\pos{\pm u}$ is at least $\rho/2>\frac13$. For such a middle third $I$, of
length $\frac{2\pi}3$, the rate $\gamma=\kappa(\theta)\pos{\pm u}$ satisfies
$2\int_I\gamma\ge\frac{4\pi}{9}K\bigl(1+\dist(0,\theta(I))^2\bigr)^2$. As
$\frac{4\pi}9>1$, Lemma~\ref{lem:cubic} applies on the half-turn with
\begin{equation}\label{eq:eta}
  (2\Gamma)^{-1/2}\le\eta(I):=K^{-1/2}\bigl(1+\dist(0,\theta(I))^2\bigr)^{-1}.
\end{equation}

\begin{lemma}\label{lem:summable}
Along any forward trajectory, the sum of $\eta(I)$ over all middle thirds is less
than $3K^{-1/2}$, whatever the initial clock state.
\end{lemma}

\begin{proof}
Consecutive middle thirds are centred at times $\pi$ apart, hence at values of
$\theta$ forming a progression $c_0+\pi\N$, and $\dist(0,\theta(I))\ge|c|-\frac\pi3$
for the centre $c$ of $I$. Among the nonnegative centres the $i$-th smallest
($i=0,1,\dots$) is at least $i\pi$, and likewise for the negative ones. Since
$i\pi-\frac\pi3\ge2i$ for $i\ge1$,
\[
  \sum_I\eta(I)\le2K^{-1/2}\Bigl(1+\sum_{i\ge1}\frac1{1+4i^2}\Bigr)
  <2K^{-1/2}\Bigl(1+\frac{\pi^2}{24}\Bigr)<3K^{-1/2}. \qedhere
\]
\end{proof}

This is where the clock $\theta$ earns its place: the gain grows along every
trajectory, and since $\theta$ passes near $0$ at most once, the bound is uniform
in $\theta(0)$. For a complete turn $j$ write $a_j=a(t_j)$, let $h_j$ be the
budget component of $a_j$, and put
\[
  E_j=\int_{t_j}^{t_j+2\pi}\nrm{\varphi(t)}\,dt,\qquad
  \varepsilon_j=2E_j+\eta(I_j^{c})+\eta(I_j^{p}),
\]
where $I_j^c$ and $I_j^p$ are the middle thirds of the compute and copy
half-turns.

\begin{lemma}[One turn]\label{lem:turn}
For every complete turn $j$ the following hold.
\begin{itemize}
\item[(i)] If $F$ is $1$-Lipschitz for $\nrm{\cdot}$ and $H(a(t))=F(a(t))$
  throughout the compute half-turn, then $\nrm{a_{j+1}-F(a_j)}\le\varepsilon_j$.
\item[(ii)] In any case, $h_{j+1}\ge h_j-\varepsilon_j$.
\item[(iii)] If $E_j\le\frac1{16}$ and the output is positive somewhere in turn
  $j$, then $\nrm{a_{j+1}-M_0(a_j)}\le\varepsilon_j$. In particular
  $h_{j+1}\ge\pos{h_j}+1-\varepsilon_j$.
\end{itemize}
\end{lemma}

\begin{proof}
Let $E^c$ and $E^p$ be the forcing masses of the two half-turns.

\emph{(i)} During the compute half-turn $a$ moves only by forcing, so
$\nrm{a(t)-a_j}\le E^c$ and hence $\nrm{H(a(t))-F(a_j)}\le E^c$.
Lemma~\ref{lem:cubic}(i), applied to each component of $b$ with target
$z_*=F_i(a_j)$, gives $\nrm{b-F(a_j)}\le2E^c+\eta(I^c)$ at the end of the
half-turn. During the copy half-turn $b$ moves only by forcing and stays within
$2E^c+\eta(I^c)+E^p$ of $F(a_j)$. Lemma~\ref{lem:cubic}(i) for $a$ then gives
$\nrm{a_{j+1}-F(a_j)}\le2E^c+\eta(I^c)+2E^p+\eta(I^p)=\varepsilon_j$.

\emph{(ii)} Argue in the same way with Lemma~\ref{lem:cubic}(ii), using
$H_h(a(t))\ge h(t)\ge h_j-E^c$ from Lemma~\ref{lem:H}(a).

\emph{(iii)} Output is paid only in the compute half-turn, and only where
$|a_q|<\frac1{16}$. Since $a_q$ moves by at most $E^c\le\frac1{16}$ during that
half-turn, $|a_q|<\frac18$ throughout it, and $H=M_0$ there by
Lemma~\ref{lem:H}(b). Apply (i) with $F=M_0$.
\end{proof}

\subsection*{Tracking an honest computation}

Fix the tolerance $\delta=2^{-8}$ and set $\Delta=2\delta+3K^{-1/2}=11\cdot2^{-10}$.
By Lemma~\ref{lem:summable}, on any time interval with forcing mass at most
$\delta$,
\begin{equation}\label{eq:budget}
  \sum_j\varepsilon_j\le\Delta,\qquad \Delta+\delta<\tfrac1{32}.
\end{equation}

\begin{lemma}[Tracking]\label{lem:track}
Consider a time interval with forcing mass at most $\delta$, containing the
complete turns $j_0,j_0+1,\dots$, and regular states $a^*_0,a^*_1,\dots$ with
$a^*_{i+1}=\M(a^*_i)$. If $\nrm{a_{j_0}-a^*_0}+\sum_{j\ge j_0}\varepsilon_j\le\Delta$,
then throughout the compute half-turn of turn $j_0+i$ we have
$\nrm{a(t)-a^*_i}<\frac1{32}$. Consequently the output vanishes during turn
$j_0+i$ if the label of $a^*_i$ is nonzero, and equals $\mu(\rho^2)\pos u$ during
its compute half-turn if that label is $0$.
\end{lemma}

\begin{proof}
We show $\nrm{a_{j_0+i}-a^*_i}\le\nrm{a_{j_0}-a^*_0}+\sum_{j_0\le j<j_0+i}\varepsilon_j
\le\Delta$ by induction. During the compute half-turn $a$ moves by at most
$\delta$, so it stays within $\Delta+\delta<\frac1{32}$ of $a^*_i$, where
$H=F_{a^*_i}$ by Lemma~\ref{lem:H}(c). Lemma~\ref{lem:turn}(i), together with
$F_{a^*_i}(a^*_i)=a^*_{i+1}$ and the Lipschitz bound, gives the induction step.
For the output, note that $\chi(4a_q)=1$ if $|a_q|\le\frac1{32}$ and
$\chi(4a_q)=0$ if $a_q$ is within $\frac1{32}$ of a label $\ge1$.
\end{proof}

Errors add up but are never amplified, because the local maps $F_{a^*}$ are
non-expansive. This is why the machine was built from translations and restarts
only, and why a summable error sequence suffices to follow infinitely many steps.

\subsection*{The two properties of the monitor}

\begin{proposition}[Uniform bound when $U$ halts]\label{prop:halt}
Suppose $U$ halts on $n$ after $T$ steps, and set $A_n=3(T+3)$. Then every
solution of~\eqref{eq:monitor} on a bounded interval $[t_0,t_1]$, from any
initial state and with any continuous forcing, satisfies
\begin{equation}\label{eq:robust}
  \int_{t_0}^{t_1}o(x(t))\,dt\le A_n+\frac{A_n}{\delta}\int_{t_0}^{t_1}\nrm{\varphi(t)}\,dt .
\end{equation}
\end{proposition}

\begin{proof}
\emph{Step 1: small forcing suffices.} Suppose~\eqref{eq:robust} is known when
the forcing mass is at most $\delta$. In general, split $[t_0,t_1]$ into at most
$1+\delta^{-1}\int\nrm{\varphi}$ subintervals of mass at most $\delta$ and apply
the bound on each. This is legitimate precisely because the bound holds for every
initial state. So assume the mass is at most $\delta$.

\emph{Step 2: counting.} If $\rho^2\notin(\frac12,\frac32)$ there is no output.
Otherwise $o\le\pos u\le\rho$, and each turn pays at most $\int\pos u=2\rho<\frac52$.
The interval consists of complete turns and at most two partial ones. Call a
complete turn \emph{productive} if it pays positive output. It suffices to show
that there are at most $T+1$ productive turns, for then
$\int o\le\frac52(T+3)\le A_n$.

\emph{Step 3: every payment raises the budget.} By Lemma~\ref{lem:turn}, a
productive turn gives $h_{j+1}\ge\max\{h_j+1,1\}-\varepsilon_j$, and any other turn
gives $h_{j+1}\ge h_j-\varepsilon_j$. Summing with~\eqref{eq:budget}: after the
$p$-th productive turn, $h\ge p-\Delta$ at every later turn boundary.

\emph{Step 4: after $T+1$ payments nothing more is paid.} Let $j'$ be the
$(T+1)$-st productive turn, if there is one. Then $h_{j'}\ge T-\Delta$, and by
Lemma~\ref{lem:turn}(iii) the state $a_{j'+1}$ lies within $\varepsilon_{j'}$ of
$a_0^*:=M_0(a_{j'})=(\bar h,\bar h,r_n,1)$ with $\bar h=\pos{h_{j'}}+1>T$. From
$a^*_0$ the machine $\M$ runs $U$ on $n$ with budget $\bar h$: during its $T$ steps
the remaining budget is at least $\bar h-T+1\ge1$, and afterwards the label is
$S+1$ forever. All these states are regular. By Lemma~\ref{lem:track}, applied
from turn $j'+1$ with initial error $\varepsilon_{j'}$, no later turn pays output.
\end{proof}

\begin{proposition}[A divergent trajectory when $U$ does not halt]\label{prop:nonhalt}
Suppose $U$ does not halt on $n$, and let $x_n^*=(0,0,-1,a_n^*,a_n^*)$ with
$a_n^*=(0,0,r_n,0)$. For every continuous forcing on $[0,\infty)$ of total mass at
most $\delta$, the solution from $x_n^*$ exists for all $t\ge0$ and
$\int_0^\infty o(x(t))\,dt=\infty$.
\end{proposition}

\begin{proof}
Global existence follows from Lemma~\ref{lem:noblowup} below. The initial clock
has $\rho=1$ and $\psi=-\frac\pi2$, so $t=0$ starts a turn. The $\M$-orbit of
$a_n^*$ consists of integer, hence regular, states and visits the restart label
infinitely often by~\eqref{eq:restarts}. By Lemma~\ref{lem:track}, with zero
initial error, every turn whose reference label is $0$ pays $\mu(1)\int\pos u=2$.
\end{proof}

\begin{lemma}[No blow-up under integrable forcing]\label{lem:noblowup}
A solution of~\eqref{eq:monitor} on $[0,t_1)$, $t_1<\infty$, with
$\int_0^{t_1}\nrm{\varphi}<\infty$ is bounded there. Hence a maximal solution
cannot end at a finite time before which its forcing mass is finite.
\end{lemma}

\begin{proof}
The clock $(\theta,u,v)$ is explicit, and $[0,t_1)$ meets finitely many
half-turns (none if $\rho=0$). On a compute half-turn $a$ moves only by forcing
and stays bounded, so the target $H(a)$ is bounded and Lemma~\ref{lem:cubic}(iii)
bounds $b$. On a copy half-turn the roles are exchanged. Proceed half-turn by
half-turn.
\end{proof}

\begin{remark}[How uniformity arises]\label{rem:mechanism}
An arbitrary initial state need not encode any computation, and the monitor may
spend a long time processing garbage. But garbage is free: output is paid only at
the restart label, and every paid restart is clean (Lemma~\ref{lem:H}(b)). It
erases the garbage and restarts the honest computation with a budget that has
never decreased (Lemma~\ref{lem:H}(a)). After $T+1$ payments the honest
computation has enough budget to reach the halting label, where it stays. The
construction never uses $T$; it appears only in the bound $A_n$.
\end{remark}

\section{Algebraization}\label{sec:roots}

The monitor is built from coordinates and rational constants by $+$, $\times$ and
$|\cdot|$. Absolute values cannot be replaced by rational functions at bounded
cost: if $|R(t)-|t||$ were bounded for a rational $R$, then $R(t)/t$ would tend to
$1$ at $+\infty$ and to $-1$ at $-\infty$, which is impossible for a rational
function. We therefore pass through square roots, which ODEs can represent with
extra variables. Both steps create errors, and both must be controlled uniformly
over all initial states, including the initial values of the new variables. The
unifying principle is Proposition~\ref{prop:halt}: \emph{every error is a forcing
of the monitor, and forcing mass costs output at the fixed rate $A_n/\delta$.}

\subsection*{Smoothing with errors integrable in time}

Call a \emph{formula} an expression built from the coordinates of $x$ and rational
constants by $+$, $\times$ and $|\cdot|$. Each component of $f$ and $o$ is given
by a fixed formula, in which $n$ enters only through constants.

\begin{lemma}\label{lem:smooth}
Let $E$ be a formula and $L\colon\R^{d_0}\to[1,\infty)$. Let $\hat E$ be obtained
by replacing, from the inside out, each subexpression $|E'|$ by
$\sqrt{\hat E'^2+L^{-2}}$. There are polynomials $U_E,V_E\ge0$, computable from $E$
and independent of $L$, such that $|E|,|\hat E|\le U_E$ and
$|\hat E-E|\le V_E/L$ pointwise.
\end{lemma}

\begin{proof}
Induction on $E$ with the rules: a coordinate $x_i$ gets $(U,V)=(1+x_i^2,0)$; a
constant $c$ gets $(1+|c|,0)$; a sum gets $(U_1+U_2,V_1+V_2)$; a product gets
$(U_1U_2,\,V_1U_2+U_1V_2)$, because
$\hat E_1\hat E_2-E_1E_2=(\hat E_1-E_1)\hat E_2+E_1(\hat E_2-E_2)$; an absolute
value gets $(U_1+1,V_1+1)$, because
$0\le\sqrt{t^2+L^{-2}}-|t|\le L^{-1}$ and $\bigl||\hat E_1|-|E_1|\bigr|\le|\hat E_1-E_1|$.
\end{proof}

Let $V=1+\sum V_E$, the sum running over the components of $f$ and $o$, and set
\[
  \epsilon_*=2^{-10},\qquad L=\epsilon_*^{-1}(1+\theta^2)\,V,\qquad
  \sigma(\theta)=\frac{\epsilon_*}{1+\theta^2}.
\]
Since $V_E\le V$, the smoothed field $\hat f$ and output $\hat o$ satisfy, on all
of $\R^{d_0}$,
\begin{equation}\label{eq:smooth}
  \nrm{\hat f-f}\le\sigma(\theta),\qquad|\hat o-o|\le\sigma(\theta),\qquad
  \int\sigma(\theta(t))\,dt\le\pi\epsilon_*<\delta,
\end{equation}
where the integral is taken along any trajectory segment; the clock contains no
absolute values, so $\dot\theta=1$ is untouched. The precision must grow with $|x|$,
to beat the polynomial growth of the $V_E$ over unbounded initial states, and with
$\theta$, to make the error integrable in time. Consequently every
$\hat f$-trajectory is a monitor trajectory with forcing
$\varphi(t)=\hat f(x(t))-f(x(t))$ of mass less than $\delta$. By
Proposition~\ref{prop:halt}, if $U$ halts then
$\B(\hat f,\hat o)\le A_n+(A_n/\delta+1)\pi\epsilon_*$. By
Lemma~\ref{lem:noblowup} and Proposition~\ref{prop:nonhalt}, if $U$ does not halt
then the $\hat f$-solution from $x_n^*$ is global and has infinite output. We
already have an analytic system with properties (H) and (N); it remains to
remove the radicals.

\subsection*{Square roots as extra variables}

List the radicals of $\hat f$ and $\hat o$ in evaluation order, $j=1,\dots,J$.
All their arguments are rational with denominators powers of $L$, and
$\sqrt{(N/L^k)^2+L^{-2}}=\sqrt{N^2+L^{2k-2}}/L^k$. So the $j$-th radical can be
written $\xi_j/L^{k_j}$, with $k_j\ge1$ and
\begin{equation}\label{eq:radicals}
  \xi_j=\sqrt{D_j(x,\xi_1,\dots,\xi_{j-1})},\qquad D_j=N_j^2+L^{2k_j-2}\ge1,
\end{equation}
where $N_j$ is a polynomial. Denote by $\xi^*(x)$ the values defined recursively
by~\eqref{eq:radicals}. Then $\hat f(x)=\bar f(x,\xi^*(x))$ and
$\hat o(x)=\bar o(x,\xi^*(x))$ for rational functions $\bar f,\bar o$ whose
denominators are powers of $L$.

Introduce variables $w\in\R^J$, meant to satisfy $w_j^2=\xi_j$, and set
\[
  \tilde f(x,w)=\bar f(x,w_1^2,\dots,w_J^2),\quad
  \tilde o(x,w)=\bar o(x,w_1^2,\dots,w_J^2),\quad
  \D_j(x,w)=D_j(x,w_1^2,\dots,w_{j-1}^2),
\]
and the \emph{defects} $R_j=w_j^4-\D_j$. The rational system $G$ on $\R^d$,
$d=d_0+J$, is
\begin{equation}\label{eq:G}
  \dot x=\tilde f(x,w),\qquad
  \dot w_j=\frac{w_j\Lambda_j}{4\D_j},\qquad
  \Lambda_j=\nabla_x\D_j\cdot\tilde f+\sum_{i<j}\partial_{w_i}\D_j\,\dot w_i,
\end{equation}
defined recursively in $j$. Its denominators are powers of $L$ and of the
$\D_j\ge1$, so $G$ is smooth on all of $\R^d$. Its clock components are exactly
$\dot\theta=1$, $\dot u=-v$, $\dot v=u$.

\begin{lemma}[First integrals]\label{lem:integrals}
Along solutions of~\eqref{eq:G}, $\Lambda_j=\frac{d}{dt}\D_j$ and each ratio
$w_j^4/\D_j$ is constant. Hence $\{R=0\}$ is invariant. On it $w_j^2=\xi_j^*(x)$
for all $j$, so the $x$-equation is $\dot x=\hat f(x)$ and $\tilde o=\hat o$.
\end{lemma}

\begin{proof}
The first claim is the chain rule, as $\D_j$ depends only on $x$ and $w_{<j}$.
Then $\frac{d}{dt}w_j^4=4w_j^3\dot w_j=(w_j^4/\D_j)\,\frac{d}{dt}\D_j$, so
$w_j^4/\D_j$ is constant. On $\{R=0\}$ we have $w_j^2=\sqrt{\D_j}$ because
$w_j^2\ge0$, and induction on $j$ gives $w_j^2=\xi^*_j(x)$.
\end{proof}

Two features deserve comment. Using $w_j^2$, rather than a variable for $\xi_j$
itself, rules out the spurious branch $-\sqrt{D_j}$, which would also be
compatible with $\xi_j^2=D_j$. And the dynamics does not correct wrong initial
values of $w$: the ratios $w_j^4/\D_j$ are conserved. We do not try to make the
auxiliary variables converge; instead we charge for their errors. For this we
need the errors to be controlled by the defects.

\begin{lemma}[Defects control the error]\label{lem:defects}
There is a polynomial $W(x,w)\ge1$, computable and of degree bounded
independently of $n$, such that on all of $\R^d$
\begin{equation}\label{eq:defects}
  \nrm{\tilde f-f}\le\sigma(\theta)+W|R|,\qquad|\tilde o-o|\le\sigma(\theta)+W|R| .
\end{equation}
\end{lemma}

\begin{proof}
By~\eqref{eq:smooth} it suffices to bound $|\tilde f-\hat f|$ and
$|\tilde o-\hat o|$ by $W|R|$. We use two elementary tools.

\emph{Tool 1.} For a polynomial $p(x,\eta)$ there are polynomials $C_i$ with
nonnegative coefficients such that
$|p(x,\eta)-p(x,\eta')|\le\sum_iC_i(x,T)|\eta_i-\eta_i'|$ whenever
$|\eta_i|,|\eta'_i|\le T_i$; this follows by telescoping,
$|s^e-t^e|\le eT^{e-1}|s-t|$ and $|x_i|\le1+x_i^2$.

\emph{Tool 2.} Polynomial majorants $S_j(x)\ge\xi^*_j(x)$: take
$S_j=1+D_j^+(x,S_1,\dots,S_{j-1})$, where $D_j^+$ is obtained from $D_j$ by
replacing each coefficient by its absolute value and each $x_i$ by $1+x_i^2$; this
works because $\sqrt D\le1+D$.

Let $d_j=|w_j^2-\xi^*_j(x)|$. Since $w_j^4-\D_j=(w_j^2-\sqrt{\D_j})(w_j^2+\sqrt{\D_j})$
and the second factor is at least $1$, we have $|w_j^2-\sqrt{\D_j}|\le|R_j|$.
Since $\sqrt{\cdot}$ is $\frac12$-Lipschitz on $[1,\infty)$,
$|\sqrt{\D_j}-\xi^*_j|\le\frac12|D_j(x,w^2_{<j})-D_j(x,\xi^*_{<j})|$. By Tool~1
with $T_i=w_i^2+S_i$,
\[
  d_j\le|R_j|+\sum_{i<j}C_{ji}(x,w)\,d_i,
\]
and induction on $j$ gives $d_j\le\sum_{i\le j}W_{ji}(x,w)|R_i|$ with polynomials
$W_{ji}\ge0$. Finally, each component of $\tilde f-\hat f$ or $\tilde o-\hat o$ has
the form $\bigl(p(x,w^2)-p(x,\xi^*)\bigr)/L^e$ with $L\ge1$, and Tool~1 bounds it
by a polynomial combination of the $d_j$. Collecting terms and using
$|R_i|\le|R|$ gives $W$. Only fixed symbolic operations are involved, so the
degree of $W$ does not depend on $n$.
\end{proof}

No proximity to $\{R=0\}$ is assumed: $w$ may be huge, and the bound is global.
The price is the polynomial weight $W$, which reappears in the penalty.

\subsection*{A quadratic penalty}

Give the rational system $G$ the output
\begin{equation}\label{eq:penalty}
  g=\tilde o-(1+\theta^2)\,W^2|R|^2 .
\end{equation}

\begin{proposition}\label{prop:G}
$U$ halts on $n$ if and only if $\B(G,g)<\infty$.
\end{proposition}

\begin{proof}
Suppose $U$ halts, and let $(x,w)$ solve~\eqref{eq:G} on $[0,t]$. Its
$x$-component solves the monitor~\eqref{eq:monitor} with the continuous forcing
$\varphi=\tilde f(x,w)-f(x)$, since the clock is exact, and
$\nrm{\varphi}\le\sigma+\mathcal R$ with $\mathcal R=W|R|$. By
Proposition~\ref{prop:halt} and~\eqref{eq:defects}, with $c_n=A_n/\delta+1$,
\begin{align*}
  \int_0^tg&\le A_n+c_n\int_0^t\sigma
      +\int_0^t\bigl(c_n\mathcal R-(1+\theta^2)\mathcal R^2\bigr)\\
  &\le A_n+c_n\pi\epsilon_*+\frac{c_n^2}{4}\int_0^t\frac{ds}{1+\theta(s)^2}
  \le A_n+c_n\pi\epsilon_*+\frac{\pi c_n^2}{4},
\end{align*}
by completing the square. The bound is uniform in the initial values of $x$ and
$w$. The quadratic penalty beats the linear cost of the defects, and the weight
$1+\theta^2$ makes the remainder integrable in time.

Suppose $U$ does not halt. Let $w_n^*$ have components
$\xi^*_j(x_n^*)^{1/2}$, so that $R=0$ at $(x_n^*,w_n^*)$. By
Lemma~\ref{lem:integrals} the solution stays in $\{R=0\}$, its $x$-component
solves $\dot x=\hat f(x)$, and $g=\hat o$ along it. The $\hat f$-solution from
$x_n^*$ is global with infinite output, as shown after~\eqref{eq:smooth}, and
$w_j^4=\D_j$ stays bounded on bounded time intervals. So the $G$-solution is
global and its output diverges.
\end{proof}

The coefficient of the penalty does not depend on the unknown halting time:
completing the square absorbs whatever $c_n$ the proof produces. This is exactly
where it matters that only upper bounds on the output are needed
(Definition~\ref{def:B}).

\subsection*{Clearing denominators}

Let $Q=L^{e}\prod_j\D_j^{e_j}\ge1$, with exponents large enough that $P=QG$ and
$b=Qg$ are polynomials.

\begin{lemma}\label{lem:timechange}
$\B(P,b)=\B(G,g)$.
\end{lemma}

\begin{proof}
If $X$ solves $\dot X=G(X)$ on $[0,t_1]$, let $t(s)$ solve $dt/ds=Q(X(t))$,
$t(0)=0$. Then $Y(s)=X(t(s))$ solves $dY/ds=P(Y)$, and
$\int b(Y)\,ds=\int Qg\,ds=\int g\,dt$. On the compact orbit segment
$1\le Q\le\max Q$, so $t(s)$ reaches $t_1$ at a finite time $s_1$. The same
argument with $ds/dt=1/Q$ works backwards. Finite segments of the two systems thus
correspond, with equal outputs.
\end{proof}

\section{Proof of the Main Theorem}\label{sec:proof}

Fix $U$ as in Fact~\ref{fact:minsky}, with $k$ registers and $S$ instruction
labels. For input $n$, Sections~\ref{sec:machine}--\ref{sec:roots} produce
$(G_n,g_n)$ and $(P_n,b_n)$. Everything except $r_n$ is independent of $n$: the
formulas for $f$ and $o$ have a fixed shape, and all later steps (the majorants
$V_E$, the scale $L$, the radicals, $W$, $Q$) are fixed symbolic operations on
these formulas. We keep every radical in the list even if a special input would
allow a simplification, so $J$ is fixed. Thus the dimension $d=3+2(k+3)+J$ is
fixed, the degrees of $P_n$ and $b_n$ are bounded independently of $n$, and the
coefficients are computable from $n$ by rational arithmetic. The constants $K$
and $\epsilon_*$ are rational; $\delta$, $\Delta$ and the witness state
$(x^*_n,w^*_n)$ appear only in proofs.

Choose once and for all an odd $\ell$ exceeding the degree bound for $P_n$, with
$\ell-1$ exceeding that for $b_n$, and let $F_n$ be the field~\eqref{eq:F} built
from $(P_n,b_n)$. It lives on $\R^N$, $N=d+1=2k+J+10$, and has odd degree
$D=\ell+2$. By Proposition~\ref{prop:G}, Lemma~\ref{lem:timechange} and
Proposition~\ref{prop:lift},
\begin{align*}
  U\text{ halts on }n
  &\iff\B(G_n,g_n)<\infty\iff\B(P_n,b_n)<\infty\\
  &\iff\text{the origin is Lyapunov stable for }\dot Y=F_n(Y).
\end{align*}
An algorithm deciding stability in this class would therefore decide the halting
problem for $U$, contradicting Fact~\ref{fact:minsky}. \qed

\paragraph{The construction at a glance.} For the record, the map $n\mapsto F_n$
consists of the following explicit steps. None of them runs $U$ or uses its
halting time.
\begin{enumerate}
\item Put $r_n=(n,0,\dots,0)$ into $M_0$ and form $H$ by~\eqref{eq:H}.
\item Form the monitor field $f$ by~\eqref{eq:monitor} with $\varphi=0$, and the
  output $o$ by~\eqref{eq:output}.
\item Compute the polynomials $V_E$ of Lemma~\ref{lem:smooth}, then $V$, $L$ and
  the smoothings $\hat f$, $\hat o$.
\item List the $J$ radicals and the $D_j$; form $\tilde f$, $\tilde o$, $\D_j$,
  $\Lambda_j$ and the system $G$ of~\eqref{eq:G}.
\item Compute $W$ (Lemma~\ref{lem:defects}) and the penalized output $g$
  by~\eqref{eq:penalty}.
\item Choose $Q$ and form $(P_n,b_n)=(QG,Qg)$.
\item Apply the projective lift~\eqref{eq:F} with the fixed odd $\ell$.
\end{enumerate}

\section{Removing all nonzero equilibria}

The constructed $F$ has an entire hyperplane of equilibria. The following
general construction removes this degeneracy without changing Lyapunov
stability. It is an explicit polynomial version of a rotating-frame change of
variables; its proof uses only homogeneity and direct differentiation.

\begin{lemma}[A unique equilibrium by rotation]\label{lem:unique}
Let $F\colon\R^{N_0}\to\R^{N_0}$ be a polynomial vector field with rational
coefficients, homogeneous of degree $D_0\ge1$. One can effectively construct a
rational homogeneous polynomial field $\mathcal F$ of degree $2D_0+1$ on
$\R^{N_1}$, where
\begin{equation}\label{eq:unique-size}
  N_1=2\left\lceil\frac{N_0}{2}\right\rceil+2,
\end{equation}
such that
\begin{equation}\label{eq:unique-equivalence}
  \mathcal F^{-1}(0)=\{0\},\qquad
  0\text{ is Lyapunov stable for }\mathcal F
  \iff 0\text{ is Lyapunov stable for }F.
\end{equation}
\end{lemma}

\begin{proof}
\emph{Even dimension.} Put $N_e=2\lceil N_0/2\rceil$. If $N_0$ is odd,
replace $F(Y)$ by $(F(Y),0)$ on $\R^{N_0+1}$. The added coordinate is constant,
so stability is preserved in both directions, and the extended field remains
homogeneous of degree $D_0$. Denote this field on $\R^{N_e}$ by $F_e$.

\emph{Construction.} Choose
\[
  J=\operatorname{diag}\left(
      \begin{pmatrix}0&-1\\1&0\end{pmatrix},\ldots,
      \begin{pmatrix}0&-1\\1&0\end{pmatrix}\right),
  \qquad J^\top=-J,\quad J^2=-I.
\]
For $X\in\R^{N_e}$ and $a,b\in\R$, set
\[
  A=aI+bJ,\qquad S=|X|^2+a^2+b^2.
\]
Define $\mathcal F$ by
\begin{equation}\label{eq:rotation}
  \begin{aligned}
    \dot X&=S^{D_0}JX+A F_e(A^\top X),\\
    \dot a&=-S^{D_0}b,\\
    \dot b&= S^{D_0}a.
  \end{aligned}
\end{equation}
Every component is a rational homogeneous polynomial of degree $2D_0+1$:
$A^\top X$ has degree two, and multiplication by $A$ adds one to the degree
$2D_0$ of $F_e(A^\top X)$. The other terms have the same degree. The last two
components ensure that the field has this degree even if $F$ vanishes
identically. All operations are effective over $\Q$.

\emph{Unique equilibrium.} At a nonzero point, $S^{D_0}>0$, so the last two
equilibrium equations force $a=b=0$. Then $A=0$, and the first equation reduces
to $|X|^{2D_0}JX=0$. Since $J$ is invertible, this forces $X=0$, a contradiction.
The origin itself is an equilibrium.

\emph{Exact dynamics in the rotating frame.} Along every solution,
$\rho^2:=a^2+b^2$ is constant and
\[
  AA^\top=A^\top A=\rho^2I,\qquad \dot A=S^{D_0}JA.
\]
For $Z=A^\top X$, direct differentiation cancels the rotation terms:
\begin{equation}\label{eq:rotation-conjugacy}
  \dot Z=\rho^2 F_e(Z),\qquad |Z|=\rho|X|.
\end{equation}
Thus, when $\rho>0$, $Z$ follows a trajectory of $F_e$ with the constant
positive change of time $\tau=\rho^2t$.

\emph{Stability is preserved uniformly in $\rho$.} Suppose that $F_e$ is stable.
By Lemma~\ref{lem:scaling}, its solutions are global and satisfy
$|Z(\tau)|\le C|Z(0)|$ for some $C\ge1$. For every solution
of~\eqref{eq:rotation} with $\rho>0$, equation~\eqref{eq:rotation-conjugacy}
therefore gives, on its maximal interval,
\[
  |X(t)|=\frac{|Z(t)|}{\rho}\le C\frac{|Z(0)|}{\rho}=C|X(0)|.
\]
The constant is independent of $\rho$, including arbitrarily small positive
amplitudes. If $\rho=0$, then $a=b=0$ and $\dot X=|X|^{2D_0}JX$, so
$\frac{d}{dt}|X|^2=0$. In either case,
\begin{equation}\label{eq:rotation-bound}
  |(X(t),a(t),b(t))|^2
  \le C^2|X(0)|^2+\rho^2
  \le C^2|(X(0),a(0),b(0))|^2.
\end{equation}
Boundedness and the polynomial continuation criterion give global existence,
and this bound proves stability of $\mathcal F$.

\emph{Instability is preserved.} Suppose that $F_e$ is unstable. There must be
finite solution segments with arbitrarily large ratios of final to initial
norm: otherwise a uniform bound on all such ratios, together with continuation,
would imply stability by Lemma~\ref{lem:scaling}. By homogeneity, rescale these
segments to obtain solutions $z_j\colon[0,T_j]\to\R^{N_e}$ satisfying
\[
  |z_j(0)|=1,\qquad |z_j(T_j)|\longrightarrow\infty.
\]
Each segment lifts to~\eqref{eq:rotation} with $\rho=1$. Explicitly, set
\[
  \alpha_j(t)=\int_0^t(1+|z_j(s)|^2)^{D_0}\,ds,\qquad
  A_j(t)=\cos\alpha_j(t)\,I+\sin\alpha_j(t)\,J,
\]
and take
\[
  X_j=A_jz_j,\qquad a_j=\cos\alpha_j,\qquad b_j=\sin\alpha_j.
\]
Here $A_j^\top X_j=z_j$, and direct differentiation verifies all three
equations~\eqref{eq:rotation}. The initial full norm is $\sqrt2$, whereas the
final full norm is $\sqrt{1+|z_j(T_j)|^2}\to\infty$. Hence $\mathcal F$ has no
uniform amplification bound and is unstable by Lemma~\ref{lem:scaling}.
\end{proof}

\begin{remark}\label{rem:rotation-periodic}
The extension adds two coordinates when $N_0$ is even and three when it is odd.
It preserves Lyapunov stability, but it does not produce asymptotic stability:
on the invariant subspace $a=b=0$, every nonzero solution has constant norm
and is periodic, with period $2\pi/|X(0)|^{2D_0}$. These periodic orbits occur
arbitrarily close to the unique equilibrium.
\end{remark}

From now on the task is to build polynomial pairs $(P_n,b_n)$ with
$\B(P_n,b_n)<\infty$ exactly when $U$ halts on $n$.

%\section{Discussion}\label{sec:discussion}
%The smallest dimension in which we obtained undecidability, via a more involved proof is $N = 5$~\cite{}. Arnold believed that the stability question may be undecidable already in dimension two or three. Proving or disproving it in these dimensions remains an open problem.

%We also obtained an analogous undecidability result for Global asymptotic stability (GAS) of equilibria of polynomial vector fields in~\cite{}. The family of polynomial vector fields constructed in~\cite{} is non-homogenous and it remains an open question whether GAS for homogenous vector fields is decidable or not.

\section{Acknowledgement}
The results in this paper were obtained using ChatGPT 6 Astra and Claude Opus 5. The author takes no intellectual credit. The Lean formalization was carried out by ChatGPT 5.6 Sol (high) and Claude Opus 5.

%\paragraph{Design choices that could be varied.} The thresholds
%$\frac18,\frac1{16},\frac1{32}$ and the constants $K$, $\epsilon_*$, $\delta$ were
%chosen for convenient arithmetic only. The gain $(1+\theta^2)^2$ can be replaced
%by $(1+\theta^2)^q$ for any $q>1$, since Lemma~\ref{lem:summable} only needs the
%errors $(2\Gamma)^{-1/2}$ to be summable uniformly in $\theta(0)$. The two-copy
%simulation could be replaced by any robust scheme with the three properties used
%above: errors are not amplified near regular states, the budget is monotone
%everywhere, and positive output forces a genuine restart.

%\paragraph{Status.} This is a reorganized and simplified presentation of a
%candidate argument, not a refereed resolution of Arnold's question. The points
%deserving the closest scrutiny are the uniform output bound
%(Proposition~\ref{prop:halt}), the global defect estimate
%(Lemma~\ref{lem:defects}), and the penalty argument (Proposition~\ref{prop:G}).

\end{document}